\documentclass[12pt]{article}

\usepackage{amsfonts}
\usepackage{amsthm}
\usepackage{mathtools}
\usepackage{geometry}
\usepackage{mathrsfs}
\usepackage{amsmath}
\usepackage{amssymb}
\usepackage{amsfonts}
\usepackage[percent]{overpic}
\usepackage{bm}
\usepackage[all]{xy}
\usepackage{graphicx}
\usepackage{subfigure}
\usepackage{latexsym}
\usepackage[colorlinks, linkcolor=blue, anchorcolor=black, citecolor=blue]{hyperref}
\usepackage{epigraph}
\usepackage{fancyhdr}
\usepackage{comment}

\usepackage{mathtools}
\mathtoolsset{showonlyrefs}

\numberwithin{equation}{section}
\usepackage{color}

\newcommand*{\medcap}{\mathbin{\scalebox{1.5}{\ensuremath{\cap}}}}
\newcommand*{\medcup}{\mathbin{\scalebox{1.3}{\ensuremath{\cup}}}}

\newtheorem{theorem}{Theorem}[section]

\newtheorem*{thma}{Theorem A}
\newtheorem{corollary}{Corollary}[section]
\newtheorem{proposition}{Proposition}[section]

\theoremstyle{remark}
\newtheorem{remark}{Remark}[section]

\newtheorem*{ack}{Acknowledgement}

\def\re{\operatorname{Re}}
\def\Im{\operatorname{Im}}
\def\area{\operatorname{area}}

\def\length{\operatorname{Length}}

\def\dens{\operatorname{dens}}

\def\diam{\operatorname{diam}}

\def\dim{\operatorname{dim}}
\def\hc{\operatorname{\widehat{\mathbb C}}}
\def\c{\operatorname{\mathbb C}}

\def\ud{\operatorname{\mathbb D}}
\def\dim{\operatorname{dim}}

\def\I{\operatorname{\mathcal{I}}}

\def\s{\operatorname{\mathcal{S}}}

\begin{document}
\title{Hausdorff dimension of escaping sets of meromorphic functions II}
\author{Magnus Aspenberg and Weiwei Cui}
\date{}
\maketitle

\begin{abstract}
	A function which is transcendental and meromorphic in the plane has at least two singular values. On one hand, if a meromorphic function has exactly two singular values, it is known that the Hausdorff dimension of the escaping set can only be either $2$ or $1/2$. On the other hand, the Hausdorff dimension of escaping sets of Speiser functions can attain every number in $[0,2]$ (cf. \cite{ac1}). In this paper, we show that number of singular values which is needed to attain every Hausdorff dimension of escaping sets is not more than $4$.


\end{abstract}

\section{Introduction and main results}

In this paper we study the dynamics of transcendental meromorphic functions. It is well known that many dynamical behaviours of the function depend, to certain extent, on the iterative behaviours of its singular values. By definition, $a\in\hc$ is a singular value of a meromorphic function $f$, if it is either a critical or asymptotic value of $f$. 


Much recent attention in transcendental dynamics is directed to the so-called \emph{Speiser class} $\s$, consisting of meromorphic functions with a finite number of singular values;  see, for instance, \cite{eremenko2, goldberg4, bishop1, bishop5}. These functions are called \emph{Speiser functions}.  Many familiar functions belong to this class, including, for example, the exponential family, the cosine family and also the tangent family. Speiser functions are studied in great details and provide dynamical behaviours similar to those of polynomial and rational maps. 

\smallskip
We intend to study functions in the Speiser class with only few singular values. Even with this restriction, the classes of functions are quite diverse. See discussion in the next section. One can thus reasonably expect that very rich and diverse dynamical behaviours could occur in this case. Our main result will confirm this in a strong sense. Recall that for a meromorphic function $f$, the \emph{escaping set} $\I(f)$ of $f$ is the set of points which tend to $\infty$ under iteration. This set plays a fundamental role in recent studies of transcendental dynamics. Starting from McMullen \cite{mcmullen11}, a wide range of research focuses on estimating the Hausdorff dimensions of escaping sets; see, for instance, \cite{baranski1, bergweiler10, rempe11, schubert} for some entire functions and \cite{bergweiler2, cuiwei2, kotus2} for certain meromorphic functions. Some of these papers also treat special Speiser functions. A natural question arises: What are the possible values of the Hausdorff dimensions of escaping sets for Speiser functions? This is resolved recently by the present authors in \cite{ac1}: Any number in $[0,2]$ can be achieved. In a larger setting (i.e., those meromorphic functions with a bounded set of finite singular values), Bergweiler and Kotus proved a similar result \cite{bergweiler2}.

\smallskip

This paper is a natural continuation of \cite{ac1} and will be focused on exploring a relation between the number of singular values and possible attainable Hausdorff dimensions for escaping sets. For convenience, we will use $\s_q$ to denote Speiser functions with exactly $q$ singular values on $\hc$. 

Our starting point is the following theorem, which collects results of several authors; see \cite{mcmullen11, kotus2} (and also \cite[Theorem 1]{cuiwei2}). By $\dim E$ we mean the Hausdorff dimension of the set $E$. 

\begin{thma}\label{pr2}
	$$\left\{\dim\I(f):\,f\in \s_2\right\}=\left\{\,1/2,\, 2\,\right\}.$$
\end{thma}

Meromorphic functions in $\s_2$ have explicit formulas; see Theorem \ref{s2} in the next section for a simple proof. As the number of singular values increases, the varieties of functions are also increasing. Thus one can reasonably expect that a more flexible result would hold. This is indeed the case, as shown by the following result.

\begin{theorem}\label{main}
	$$\left\{\dim\I(f):\,f\in \s_4\right\}=[0,2].$$
\end{theorem}

The existing gap between Theorem A and Theorem \ref{main} is the class $\s_3$. It is plausible that the above theorem holds in $\s_3$. However, our construction will not give this.

Another point that we would like to address concerns a question for the invariance of Hausdorff dimensions of escaping sets. To be more specific, it asks whether two quasiconformally equivalent functions will have the escaping sets of equal Hausdorff dimension. We say that two Speiser functions $f$ and $g$ are quasiconformally equivalent if there are quasiconformal mappings $\varphi,\psi:\c\to\c$ such that $\varphi\circ f=g\circ\psi$; see \cite{eremenko2}. (We can also define topological equivalence by requiring $\varphi$ and $\psi$ to be homeomorphisms.) The above question was originally asked for entire functions which are not necessarily Speiser functions. Counterexamples are given recently in the meromorphic setting in \cite{ac1}. Here we provide another class of counterexamples.

\begin{theorem}\label{incm}
There exist quasiconformally equivalent meromorphic functions $f,g\in\s_4$ for which $\dim\I(f)\neq\dim\I(g)$.
\end{theorem}

We also remark that if two quasiconformally equivalent meromorphic functions belong to $\s_3$, then they are actually conformally equivalent (see \cite[Lemma 2.3]{bergweiler-cui}). This implies immediately that two such functions have the same order of growth. It is plausible that two such functions will have escaping sets of the same Hausdorff dimension. In this sense, the above Theorem \ref{incm} may be optimal.

\medskip
\noindent{\emph{Structure of the article.}} In Section \ref{section2} we discuss briefly transcendental meromorphic functions with few singular values. In Section \ref{thecon} we construct meromorphic functions with four singular values of arbitrary order. The last section is then devoted to the outline of the estimate for the Hausdorff dimension of their escaping sets. 

\begin{ack}
We would like to thank the referee for many useful comments and corrections. The second author would also like to thank Vergstiftelsen for financial support.
\end{ack}

\section{Speiser functions with few singular values}\label{section2}

Let $f:\c\to\hc$ be transcendental and meromorphic. We say that $c$ is a \emph{critical value} of $f$, if $c$ has a preimage with zero spherical derivative. With this definition, $\infty$ will be a critical value if there are any multiple poles. $a\in\hc$ is an \emph{asymptotic value} of $f$, if there exists a curve $\gamma$ tending to $\infty$ such that $f(\gamma)$ tends to $a$. As a simple example, $0$ and $\infty$ are asymptotic values of $e^z$. A value $s$ is called a \emph{singular value}, if it is either a critical or asymptotic value. See \cite{bergweiler18} for a classification of singularities of the inverse of a meromorphic function. Singular values play a vital role in the dynamics of meromorphic functions, we refer to \cite{bergweiler1} for more details and explanations.

\medskip
\noindent{\emph{Meromorphic functions with two singular values.}} The following simple fact concerning meromorphic functions with two singular values is folklore and it is not easy to locate a reference. Therefore, an outline of proof is presented for completeness. Recall that $\s_q$ denote the class of Speiser functions with exactly $q$ singular values.

\begin{theorem}\label{s2}
Let $f\in\s_2$. Then $f$ is of the form $M\circ\exp\circ A$, where $M$ is M\"obius and $A$ is linear.
\end{theorem}

\begin{proof}[Sketch of proof]
Without loss of generality, we assume that the two singular values of $f$ are $0$ and $\infty$. Otherwise we consider $M_1\circ f$, where $M_1$ is a M\"obius transformation sending the two singular values of $f$ to $0$ and $\infty$. Then
$$f:\c\setminus\{f^{-1}(0),\,f^{-1}(\infty)\}\rightarrow \hc\setminus\{0,\infty\}$$
is a covering map. Note that the fundamental group of $\hc\setminus\{0,\infty\}$ is isomorphic to $\mathbb{Z}$. The transcendence of $f$ then implies that the fundamental group of $\c\setminus\{f^{-1}(0),\,f^{-1}(\infty)\}$ is trivial and thus $\c\setminus\{f^{-1}(0),\,f^{-1}(\infty)\}$ is simply connected such that the above $f$ is a universal covering to $\hc\setminus\{0,\infty\}$. Since meromorphic functions are discrete maps, so we have that $\c\setminus\{f^{-1}(0),\,f^{-1}(\infty)\}=\c$. This implies that both singular values are actually omitted. Note that the exponential map $\exp$ is a holomorphic universal covering from $\c$ to $ \hc\setminus\{0,\infty\}$. Now it follows from the essential uniqueness of the universal covering spaces that there exists a holomorphic homeomorphism $\phi: \c\to\c$ such that $f=\exp\circ\phi$. Since a holomorphic homeomorphism of the complex plane must be a linear map, we thus have
$$f(z)=e^{Az+B},$$
where $A(\neq 0),\,B$ are complex constants. This completes the proof.
\end{proof}

It follows immediately that if $f\in\s_2$, then the singular values of $f$ are both asymptotic values. Moreover, if one of the asymptotic values is at $\infty$, then $f$ is of the form $\lambda e^{z}$; if both of them are finite, then $f$ can be written as $M(e^z)$, where $M$ is M\"obius sending $0$ and $\infty$ to two finite points. This, together with the results obtained in \cite{mcmullen11, kotus2}, shows that Theorem A holds; see also \cite{cuiwei3}.

\begin{figure}[htbp] 
	\centering
	\includegraphics[width=7cm]{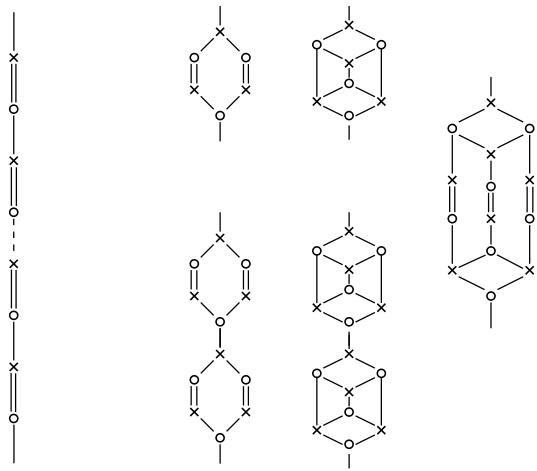}
	\caption{Producing meromorphic functions in $\s_3$ by using line complex. Replacing the dashed line in the left picture by any one of the finite graphs on the right or any finite combination of these graphs will give functions with three singular values which are not topologically equivalent.}
	\label{s3}
\end{figure}

The next natural step would be asking if a similar result stated in Theorem \ref{s2} will hold if one has more singular values. This need not be true. In fact, even for meromorphic functions with three singular values, one cannot expect a similar function-theoretic rigidity as Theorem \ref{s2}. Recent work of Bishop \cite{bishop1} shows that there are uncountably many \emph{essentially different} entire functions with two critical values. More precisely, the classes of topologically equivalent entire functions with two critical values are uncountable. See also \cite{cuiwei3}. One can also use the so-called Maclane-Vinberg method to construct entire functions with two singular values; see \cite[Observation 5.2]{bergweiler6}. For general meromorphic functions with three singular values, one can resort to the theory of line complex to construct such functions; see \cite[Chapter 7]{goldbergmero}. Without giving a detailed account of line complex, we mention that one can show that there exist infinitely many non-equivalent meromorphic functions in $\s_3$ by modifying the left graph in Figure \ref{s3}: Replacing the dashed line by one of finite graphs shown on the right of Figure \ref{s3} will produce meromorphic functions with three singular values. It is clear that there are infinitely many such functions, since one can consider any finite combination of these finite graphs which are then used to replace the dashed line. The produced meromorphic functions are those with rational Schwarzian derivatives. For a complete treatment of these functions, we refer to \cite[Chapter 7]{goldbergmero} and \cite{elfving1}.

\section{The construction}\label{thecon}



For $\delta\in(0,2\pi]$, put
$$\c_{\delta}=\left\{z=re^{i\theta}:\,r>0,\,0<\theta<\delta \right\}.$$
In particular, if $\delta=2\pi$, then $\c_{\delta}$ is the slit plane $\c\setminus\mathbb{R}^+$. Choosing the natural branch of the logarithm, for any $\alpha\in (0,2\pi]$, we set
\begin{align*}
h:\,\c_{2\pi}&\to\c_{\alpha},\\
z&\mapsto z^{\alpha/2\pi}
\end{align*}
which then defines a conformal map. 

To construct meromorphic functions in $\s_4$, we will consider the restriction of some carefully chosen Weierstra{\ss} elliptic function $\wp$ in the sector domain $\c_{\alpha}$. Then the function defined as $\wp\circ h$ will be meromorphic in $\c_{2\pi}$. We will then need to extend this function across the positive real axis in order to have a meromorphic function in the plane. However, the extension need not be continuous; in other words, for $x\in\mathbb{R}^+$, $\wp\circ h (x)$ need not coincide when one approaches $x$ respectively from the upper and lower half planes. This will be resolved by introducing a spiral map and then using a quasiconformal surgery. See Figure \ref{idea} for an illustration.

\begin{remark}
The spiral map we are going to introduce will help us to achieve every possible finite order. In a recent paper \cite{bergweiler23}, Bergweiler and Eremenko also used this idea to solve an open problem in the theory of complex differential equations.
\end{remark}

\begin{figure}[htbp] 
	\centering
	\includegraphics[width=11cm]{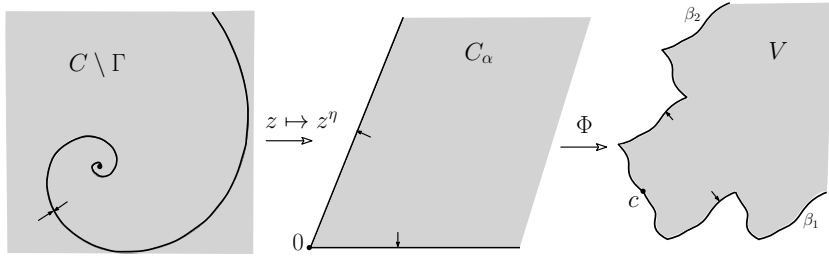}
	\caption{Shown is the idea of construction. We have put $\eta=\alpha\mu/(2\pi)$. The map $\Phi$ is quasiconformal which helps to remove the discontinuity arising on the logarithmic spiral $\Gamma$.}
	\label{idea}
\end{figure}

\medskip

From now on we fix one $\alpha$ and use the corresponding function $h$ as defined above. Then in the sector $\c_{\alpha}$ we consider a Weierstra{\ss} elliptic function $\wp$ with two periods $1$ and $\tau$, where $\tau$ will satisfy the following condition:
\begin{equation}\label{anre}
\arg\tau=
\begin{cases}
	~\alpha,&\text{if~}\,~\alpha<\pi;\\
	~\alpha-\pi,&\text{if~}\,~\alpha>\pi.
\end{cases}
\end{equation}

In case that $\alpha=\pi$, the function defined as $\wp\circ h$ extends continuously across the positive real axis and thus gives a function meromorphic in the plane (which is actually $\wp(\sqrt{z})$). This is an easy case to deal with, so we will assume in the sequel that $\alpha\neq\pi$. The main point of the condition \eqref{anre} is to make sure that the two prime periods of the chosen Weierstra{\ss} $\wp$-function lie on two boundaries of $\c_{\alpha}$. Now we put
$$e_1=\wp\left(\frac{1}{2}\right),\,\,e_2=\wp\left(\frac{1+\tau}{2}\right),\,\,e_3=\wp\left(\frac{\tau}{2}\right),$$
which are \emph{finite} critical values of $\wp$ (with another critical value at $\infty$).

One can then see immediately that
$$g_1:=\wp\circ h:\,\c_{2\pi}\to\c$$
is well defined and meromorphic in the slit plane. However, $g_1$ may not be able to extend continuously across the positive real axis, as we mentioned above. One of the main objectives in the construction we will make is to circumvent this problem.

By choosing the natural branch of the power map, let 
\begin{align}
p:  \c_{2\pi}  &\to \c; \, ~\,z\mapsto z^{1/\mu}.
\end{align}
It can readily be seen that if $\mu$ has real part equal to $1$ then the image of $\c\setminus\mathbb{R}^+ $ is an open set whose complement is a logarithmic spiral. Let us denote this spiral by $\Gamma$. Now note that if $\{x_n\}_{n=0}^{\infty}$ and $\{y_n\}_{n=0}^{\infty}$ are two sequences of complex numbers approaching a point $z \in \mathbb{R}^+$ from different sides of the real axis, the limits of $p(x_n)$ and $p(y_n)$ as $n \rightarrow \infty$ may well be different. So the map $p$ may not have a continuous extension to $\mathbb{R}^+$. For our purposes, we put
\begin{equation}\label{mudef}
	\mu=1-i\,\frac{\log|\tau|}{\alpha}.
\end{equation}

Denote by $q$ the inverse of $p$; i.e., $q(z)=z^{\mu}$. Then the following function
\begin{align}
g_2: \c\setminus\Gamma &\to\c_{2\pi},\\
z&\mapsto g_1(q(z))
\end{align}
is well defined and meromorphic in the complement of a logarithmic spiral $\Gamma$. It may not extend continuously across $\Gamma$, but by the choice of $\mu$ and $\tau$, with a simple computation, we can extend $g_2$ continuously to a discrete set of $\Gamma$ whose points will be mapped by $h\circ q$ to poles of $\wp$.
We show below how to remove the discontinuities between these discrete set of points on $\Gamma$ by using a quasiconformal surgery. In short, we will construct a quasiconformal self-map $\Phi$ of $\c_{\alpha}$ such that the new defined function $\wp\circ\Phi\circ h\circ q$ extends continuously across $\Gamma$.

\medskip
Let $P$ denote the parallelogram formed by four vertices $0,\,1/2,\,(1+\tau)/2$ and $\tau/2$. It follows from the basic properties of Weierstra{\ss} elliptic functions that $\delta:=\wp(\partial P)$ is a simple closed curve on $\hc$, which passes through $\infty$ on both sides such that $\hc\setminus\delta$
consists of two domains $A$ and $B$. Suppose without loss of generality that $A=\wp(P)$. Note that all critical values of $\wp$ lie on $\delta$. Now we can choose an analytic closed curve $\gamma$ in $\hc$ such that $\gamma\cap\delta=\{e_2,\,\infty\}$. To achieve this, we first choose an analytic curve $\gamma_1$ lying entirely in $A$ with two endpoints being $e_2$ and $\infty$; similarly an analytic curve $\gamma_2$ is chosen to lie completely in $B$ with two endpoints $e_2$ and $\infty$. Then $\gamma$ is defined as the union of $\gamma_1$ and $\gamma_2$ together with their common endpoints. See Figure \ref{onep} for an illustration. Now we consider suitable preimage of $\gamma$ under the function $\wp$. More precisely, we have the following result; compare this with Proposition 3.1 of \cite{ac1}. Put 
$$c=\dfrac{1+\tau}{2}.$$

\begin{figure}[htbp] 
	\centering
	\includegraphics[width=14cm]{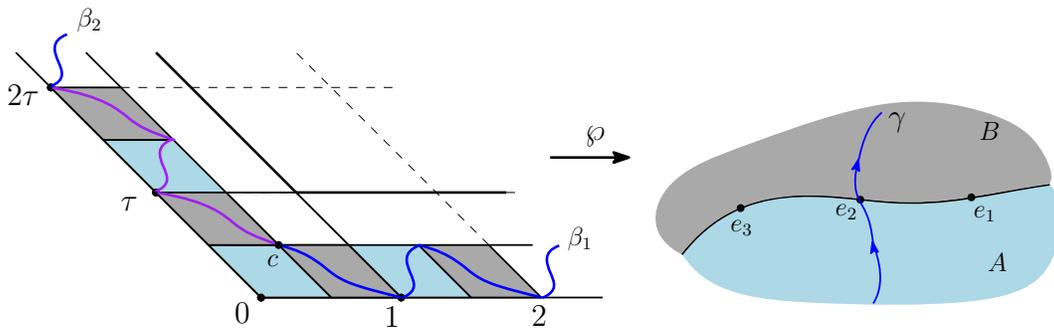}
	\caption{The curve $\gamma$ is chosen such that it passes through the critical value $e_2$ and $\infty$. Then we choose suitable two preimages $\beta_1$ and $\beta_2$ of $\gamma$ such that they are periodic and both start from $c$ which is a preimage of the critical value $e_2$.}
	\label{onep}
\end{figure}

\begin{proposition}
Let $\gamma$ be as above. Then there exist two piecewise analytic curves $\beta_i$  with $\wp(\beta_i)=\gamma$ for all $i$, such that the following hold:
\begin{itemize}
\item $\beta_1\cap\beta_2=\{c\}$.
\item $\beta_1$ starts from  the point $c$ and is periodic with period $1$, i.e., $z\in\beta_1$ implies that $z+1\in\beta_1$. Moreover, $\beta_1$ passes through poles at $n$ of $\wp$ for all $n\geq 1$.
\item $\beta_2$ starts from $c$ and is periodic with period $\tau$, i.e., $z\in\beta_2$ implies that $z+\tau\in\beta_2$. Moreover, $\beta_2$ passes through poles $n\tau$ for all $n\geq 1$.
\end{itemize}
\end{proposition}

\begin{proof}
	For convenience, we put $\gamma_1=\gamma\cap A$ and $\gamma_2=\gamma\cap B$. We will use $W+c$ as the translation by a complex number $c$ for any set $W\subset\c$; in other words, $W+c=\{z+c: z\in W\}$. Recall that $P$ is the parallelogram formed by four points $0,\,1/2,\,(1+\tau)/2$ and $\tau/2$. Now by elementary properties of Weierstra{\ss} elliptic functions we know that $\wp: P+\frac{1}{2}\to B$ is conformal. Therefore, $\gamma_2$ has a preimage, denoted by $\beta^{1}_1$ in $P+1/2$ which is an analytic curve connecting $(1+\tau)/2$ and $1$. The latter property follows easily since $\gamma_2\subset B$ connects $e_2$ and $\infty$. Similarly, since $\wp: P+1\to A$ is conformal we see immediately that $\gamma_1$ has a preimage $\beta^{2}_1$ in $P+1$ which is an analytic curve connecting $1$ and $(3+\tau)/2$. Now by periodicity of $\wp$, the curve $\beta^{1}_1\cup\beta^{2}_1$ and its translations by $n\in\mathbb{N}$ will be mapped conformally onto $\beta$. Now we define 
	$$\beta_1=\bigcup_{n\geq 0}\left(\beta^{1}_1\cup\beta^{2}_1 + n\right).$$
It is clear that $\beta_1$ satisfies all the required properties.
		
In the same way we can define $\beta_2$ as claimed. We omit details here.
\end{proof}

We denoted by $V$ the domain contained in $\c_{\alpha}$ and bounded by $\beta_1$ and $\beta_2$. Then the following is concentrated on the construction of a quasiconformal mapping
$$\Phi: \c_{\alpha}\to V$$
such that the function $\wp\circ\Phi\circ h\circ q$ extends continuously across the logarithmic spiral $\Gamma$ and thus gives a function continuous throughout the whole plane. See Figure \ref{figshow} for the construction of $\Phi$.

\begin{figure}[htbp] 
	\centering
	\includegraphics[width=13cm]{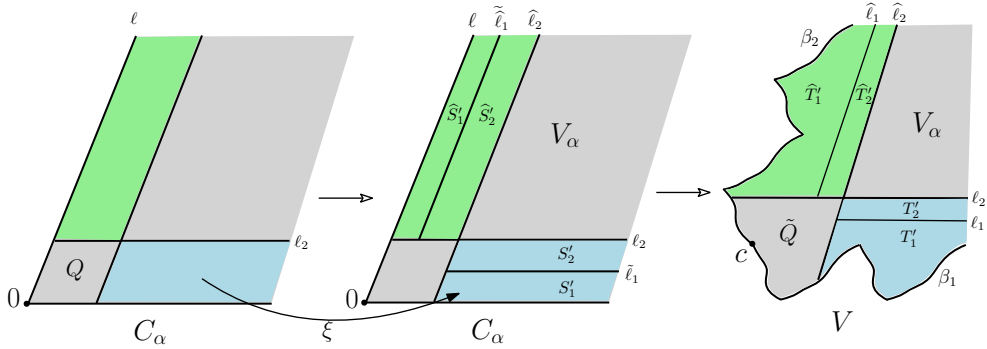}
	\caption{Show is a sketch of the construction of the quasiconformal map $\Phi$. In the figure, $S'_1$ and $S'_2$ are the restrictions of $S_1$ and $S_2$ to the sector $\c_{\alpha}$ while $T'_1$ and $T'_2$ are the corresponding images of $S'_1$ and $S'_2$ under the maps $\phi_{1,1}$ and $\phi_{1,2}$. Same applies to $\widehat{S}'_1$, $\widehat{S}'_2$ and $\widehat{T}'_1,\,\widehat{T}'_2$.}
	\label{figshow}
\end{figure}

First by periodicity of Weierstra{\ss} elliptic functions and the construction of $\beta_1$, we can define a periodic curve $\tilde{\beta}_1$ which is the extension of $\beta_1$ along the direction of the negative real axis. Now choose a real number $a>0$ such that $a>\max_{z\in\tilde{\beta}_1}\Im(z)$. Define
$$T_1=\left\{z:\,\Im_{\substack{w\in\tilde{\beta}_1}, \,\substack{\re w=\re z}}(w)<\Im(z)<a\right\}.$$
Then by \cite[Lemma 3.1]{ac1}, there is a real number $a'>0$ such that with
$$S_1=\left\{z:0<\Im(z)<a'\right\}$$
there is a conformal map
$$\phi_{1,1}: S_1\to T_1$$
which fixes three boundary points $0$ and $\pm\infty$ and is periodic with period one. In other words,
$$\phi_{1,1}(z+1)=\phi_{1,1}(z)+1$$
for any $z\in S_1$. Moreover, $\phi_{1,1}$ extends to the boundary as a piecewise diffeomorphism.

Choose another real number $b$ satisfying $b>\max\{a,\,a'\}$ and put
$$S_2=\left\{z: a'<\Im(z)<b\right\}$$
and
$$T_2=\left\{z: a<\Im(z)<b\right\}.$$
We now want to define a quasiconformal map from $S_2$ to $T_2$ which interpolates between the identity map on the upper boundary of $S_2$ and the extension of $\phi_{1,1}$ on the lower boundary of $S_2$. For convenience, we denote by $\ell_1$ the horizontal curve $\{z: \Im(z)=a\}$ and $\tilde{\ell}_1$ the curve $\{z: \Im(z)=a'\}$. We also put $\ell_2=\{z:\Im(z)=b\}$. Suppose that $\chi_1$ is the boundary extension of $\phi_{1,1}$ to $\tilde{\ell}_1$. Then it follows from the construction of $\phi_{1,1}$, and the Schwarz reflection principle (and the fact that $\tilde{\ell}_1$ and $\ell_1$ are straight lines), that $\chi_1:\tilde{\ell}_1\to\ell_1$ is analytic. Moreover, $\chi_1(z+1)=\chi_1(z)+1$ for $z\in\mathbb{R}$.
By considering
$$\widetilde{\chi}_1(z)=\chi_1(z+ia')-ia$$
we see that $\widetilde{\chi}_1: \mathbb{R}\to\mathbb{R}$ is increasing and analytic.
We also define
$$\widetilde{\chi}_2(z)=z+i(a'-a).$$
Note that $\widetilde{\chi}_1$ is obtained from $\chi_1$ and the identity map by moving the lower boundaries of $S_2$ and $T_2$ to the real axis. Suppose the new strips are $S''_2$ and $T''_2$ respectively. Then
$$L(x+iy)=\left(1-\frac{y}{b-a'}\right)\widetilde{\chi}_1(x)+\frac{y}{b-a'}\widetilde{\chi}_2(x)+i\,\frac{b-a}{b-a'}y$$
is the linear interpolation between $\widetilde{\chi}_1$ and $\widetilde{\chi}_2$. This is actually a quasiconformal map, as can be seen by checking the Jacobian of $L$, which is
$$\frac{b-a}{b-a'}\left(\left(1-\frac{y}{b-a'}\right)\widetilde{\chi}'_1(x)+\frac{y}{b-a'}\widetilde{\chi}'_2(x)\right).$$
It is strictly bigger than zero since both $\widetilde{\chi}_1$ and $\widetilde{\chi}_2$ are increasing. Together with the periodicity, this implies that $L$ is quasiconformal. Now we define
\begin{align}
	\phi_{1,2}: S_2 &\to T_2,\\
	z&\mapsto L(z-ia')+ia.
\end{align}
This is a quasiconformal map.

Along the $\tau$-direction, we use a similar idea as above to obtain a quasiconformal map. More precisely, we first consider a periodic curve $\tilde{\beta}_2$ which is the natural extension of $\beta_2$. With
$$\ell=\left\{z: \arg(z)=\arg(\tau)~\text{or}~\arg(\tau)-\pi\right\},$$
we define, for some real numbers $c',\,c$ and $d$, where $c= (1+\tau)/2$ and $c',d$ will be determined later,
\begin{align}
\widetilde{\widehat{\ell}}_1&=\ell+c', \nonumber\\
\widehat{\ell}_1&=\ell+c, \nonumber \\
  \widehat{\ell}_2&=\ell+d.
\end{align}
We also denote by $\widehat{S}_1$ the strip between $\ell$ and $\widetilde{\widehat{\ell}}_1$, by $\widehat{T}_1$ the domain bounded by $\tilde{\beta}_2$ and $\widehat{\ell}_1$. Moreover, $\widehat{S}_2$ will be the strip bound by $\widetilde{\widehat{\ell}}_1$ and $\widehat{\ell}_2$, while $\widehat{T}_2$ is bound by $\widehat{\ell}_1$ and $\widehat{\ell}_2$.

Now $c'$ and $c$ are chosen (similarly as $a'$ and $a$ above) such that there exists a conformal map
$$\phi_{2,1}: \widehat{S}_1\to\widehat{T}_1$$
which is periodic with period $\tau$ and fixes three boundary points $0$ and $\pm\infty$ (here $\pm\infty$ are understood as infinity along two directions of $\ell$). In the same way for which the map $\phi_{1,2}$ was defined, we can define a quasiconformal map
$$\phi_{2,2}: \widehat{S}_2\to\widehat{T}_2.$$
We omit details here. For later purposes, we put
$$S=\overline{S_1\medcup S_2}\medcap\c_{\alpha},\,\quad~\,T=\overline{T_1\medcup T_2}\medcap V,$$
and
$$\widehat{S}=\overline{\widehat{S}_1\medcup \widehat{S}_2}\medcap\c_{\alpha},\,\quad~\, \widehat{T}=\overline{\widehat{T}_1\medcup \widehat{T}_2}\medcap V.$$
Moreover,
$$V_{\alpha}=\c_{\alpha}\setminus \left(S\medcup\widehat{S}\right) =\c_{\alpha}\setminus \left(T\medcup\widehat{T}\right).$$

Now we consider restrictions of the above constructed maps on the domain $\c_{\alpha}$. More precisely, we define
\begin{equation}\label{phi1}
\begin{aligned}
	\phi_{1}: S &\longrightarrow T;~\,~	z\longmapsto 
	\begin{cases}
~\phi_{1,1}(z),\,&\text{if~}z\in S_1,\\
~\phi_{1,2}(z),\,&\text{if~}z\in S_2.
\end{cases}
\end{aligned}
\end{equation}
In a similar way, we have
\begin{equation}\label{phi2}
\begin{aligned}
	\phi_{2}: \widehat{S} &\longrightarrow \widehat{T};~\,~z\longmapsto 
	\begin{cases}
~\phi_{2,1}(z),\,&\text{if~}z\in \widehat{S}_1,\\
~\phi_{2,2}(z),\,&\text{if~}z\in \widehat{S}_2.
\end{cases}
\end{aligned}
\end{equation}
Finally, we define
\begin{equation}\label{phi3}
\begin{aligned}
	\phi_{3}: V_{\alpha} &\longrightarrow V_{\alpha},\\
	z&\longmapsto z.
\end{aligned}
\end{equation}

Note that the half-strips $S$ and $\widehat{S}$ are overlapped on a parallelogram 
\begin{equation}\label{omitpara}
Q=S\cap \widehat{S}
\end{equation}
which contains the origin on the boundary. Both $\phi_1$ and $\phi_2$ are defined on $Q$, but they do not necessarily coincide there. Away from $Q$, in $\c_{\alpha}\setminus Q$ we have a well defined quasiconformal map, which is $\phi_1$ in $S\setminus Q$, $\phi_2$ in $\widehat{S}\setminus Q$ and identity on $V_{\alpha}$. However, this quasiconformal map still does not satisfy our requirement to remove the aforementioned discontinuity. To proceed, we need to change $\phi_1$ or $\phi_2$ further. We will change $\phi_1$ in the following.

To this aim, we first define suitable inverse branches of $\wp$. Note that $\wp(\tilde{\beta}_i)=\gamma$. Let $\psi_i$ denote an inverse branch of $\wp$ from $\gamma$ into $\tilde{\beta}_i$, for $i=1,2$.
Recall that $p: \c_{2\pi} \mapsto \c \setminus \Gamma$, where $p(z) = z^{1/\mu}$. Put $\tilde{p}(z) = p(z^{2\pi/\alpha})$ which maps $\c_{\alpha}$ onto $\c \setminus \Gamma$.  We can extend $\tilde{p}$ to its boundary continuously. Moreover, from the definition of $\mu$, we have that $x\in \mathbb{R}$ and
$\tau x \in \ell = \{ z: \arg(z) = \arg(\tau) \}$ both are mapped by $\tilde{p}$ onto the same point in $\Gamma$. So, in a sense, the ``transition function'' $t(x) = \tau x$ for $x \in \mathbb{R}$, identifies points on the lines bounding $\c_{\alpha}$ which corresponds to the same origin in $\Gamma$. Now the desired ``correction function'' is defined as
\begin{equation}\label{corr}
	\begin{aligned}
		\kappa:\, \mathbb{R} &\longrightarrow\mathbb{R},\\
		x&\longrightarrow \phi_{1}^{-1}\circ\psi_1\circ\wp\circ\phi_{2}(\tau x).
	\end{aligned}
\end{equation}
Since $\tau x \in\ell$, the above function $\kappa$ is well defined. Roughly speaking, the function $\kappa$ fixes the difference between $\phi_{1}$ and $\phi_{2}$. With this function, we can now remove the discontinuity by considering a linear interpolation between $\kappa$ on the real axis and the identity on the horizontal line $\ell_2$. More precisely, we define
\begin{equation}
	\xi_1(x+iy)=\left(1-\frac{y}{b}\right)\kappa(x)+\frac{y}{b}\,x+i\,y~\,\text{~for~}~\,0\leq y\leq b.
\end{equation}
This map is quasiconformal, as one can check that the Jacobian of $\xi$ is non-zero almost everywhere. Moreover, we define
\begin{equation}
	\xi_2(x+iy)=x+i\,y~\,\text{~for~}~\,y\geq b.
\end{equation}
In this way, we have just constructed a quasiconformal map of the upper half-plane by setting
\begin{equation}\label{xi}
	\begin{aligned}
		\xi: \mathbb{H}^+ &\longrightarrow \mathbb{H}^+;~\,~\, z\longmapsto 
		\begin{cases}
			~\xi_1(z),\,&\text{if~} 0\leq\Im(z)\leq b;\\
			~\xi_2(z),\,&\text{if~} \Im(z)\geq b.
		\end{cases}
	\end{aligned}
\end{equation}

Now we still need to define a map on the parallelogram $Q$. Denote by $I_1$ the side of $Q$ on real axis (i.e., $I_1=[0,d]$), by $I_2$ the side on the line $\ell$ (i.e., $I_2=\{re^{i\alpha}: 0\leq r\leq b/\sin(\alpha)\}$). The other two sides of $Q$ are denoted by $I_3$ which is the one parallel to $I_1$, and $I_4$ the one paralleling to $I_2$. Now we put $\hat{I}_1=\xi(I_1)$, $\hat{I}_2=I_2$, $\hat{I}_3=I_3$ and finally $\hat{I}_4=\xi(I_4)$. Then $\hat{I}_i$, for $i=1,\dots,4$, form a quadrilateral, denoted by $\hat{Q}$. Now we continue to define a new quadrilateral $\tilde{Q}$ formed by $\tilde{I}_i$ for $i=1,\dots,4$, where
$$\tilde{I}_i=\phi_{1}\left(\hat{I}_i\right)\,~\text{for}~\,i=1,\,4,$$
and
$$\tilde{I}_i=\phi_{2}\left(\hat{I}_i\right)\,~\text{for}~\,i=2,\,3.$$
Now we can define a boundary map between $Q$ and $\tilde{Q}$ by using $\phi_1\circ\xi$ on $I_1\cup I_4$ and $\phi_{2}\circ\xi$ on $I_2\cup I_3$. The boundary map extends to the interior of $Q$ quasiconformally; see \cite[Lemma 2.24]{branner3}. So we have a quasiconformal map
$$h: Q\to\tilde{Q}.$$

Now we can define our promised map
\begin{align}\label{mapfin}
	\Phi :~ \c_{\alpha}\,&\longrightarrow\, V;~\,~
	z\longmapsto
	\begin{cases}
		~\phi_{1}(\xi(z)),\,&\text{if~}z\in S\setminus Q,\\
		~\phi_{2}(z),\,&\text{if~}z\in \widehat{S}\setminus Q,\\
		~h(z),\,&\text{if~}z\in Q,\\
		~z,\,&\text{if~}z\in V_{\alpha}.
	\end{cases}
\end{align}
One can then check that the map 
\begin{equation}\label{defG}
G(z):=\wp\circ\Phi\circ h\circ q(z)
\end{equation}
extends continuously across the logarithmic spiral $\Gamma$ and thus gives us, by construction, a quasi-meromorphic function of the plane. 

So by the measurable Riemann mapping theorem (cf. \cite{ahlfors8}) there exist a quasiconformal homeomorphism $\Psi$ and a meromorphic function $f$ such that $G=f\circ\Psi$. By our construction, $G$ is quasiconformal only in $(h\circ q)^{-1}(W)$, where $W=\c_{\alpha}\setminus V_{\alpha}$.

It follows from our construction that $f$ has exactly $4$ critical values and no asymptotic values. In other words, $f\in\s_4$. Moreover, $f$ has only double poles.

\medskip
To derive some asymptotic properties of $f$, we will need the well known Teichm\"uller-Wittich-Belinskii theorem concerning conformality of a quasiconformal mapping at a point. We refer to \cite{ahlfors8, lehto1} for a background on quasiconformal mappings and also to this result. For our purpose, a stronger result would suffice. We first recall some relevant notions. Let $\varphi:\c\to\c$ be a quasiconformal mapping, the dilatation of $\varphi$ at a point $z$ is
$$K_{\varphi}(z):=\frac{|\varphi_z|+|\varphi_{\bar{z}}|}{|\varphi_z|-|\varphi_{\bar{z}}|}=\frac{1+|\mu_{\varphi}|}{1-|\mu_{\varphi}|},$$
where $\mu_{\varphi}:=\varphi_{\bar{z}}/\varphi_{z}$ is the complex dilatation of $\varphi$.

The above mentioned result is stated as follows; see \cite[Lemma 2.2]{ac1}.

\begin{proposition}\label{stwb}
Let $\varphi:\c\to\c$ be quasiconformal. Put $A=\{z\in\c:\mu_{\varphi}\neq 0\}$. If $\iint_{A\setminus\ud} dxdy/(x^2+y^2)<\infty$. Then $\varphi(z)\sim z$ as $z\to\infty$. Upon normalisation, one has $\varphi(z)=z+o(z)$ as $z\to\infty$.
\end{proposition}

Note that $W$ is the union of two half-strips which means that the set $W\setminus\ud$ has finite logarithmic area, i.e., the integration above holds. So $\Psi$ satisfies the conditions in Proposition \ref{stwb}. So we have, up to normalization, 
\begin{equation}\label{iden}
\Psi(z)=z+o(z)\,~\,\text{as}~z\to\infty.
\end{equation}

\medskip
\noindent{\emph{Distribution of poles of $G$.}} Now fix $R>0$ large. We will be interested in counting the number of poles in the closed disk $\overline{D}(0,R)$. This will follow from the following area formula for a domain bounded by two logarithmic spirals. More precisely, let $\sigma$ be a (non-zero) complex number and $R>0$. Put $\psi(z)=z^{\sigma}$ for $z\in\c_{2\pi}$ with chosen principle branch. Let $\Gamma_{\beta}$ and $\Gamma_{\gamma}$ be the $\psi$-images of two radial lines of arguments $\beta$ and $\gamma$ respectively, where $0\leq\beta<\gamma\leq 2\pi$. See Figure \ref{area}. Denote by $A:=A_{\gamma-\beta}$ the bounded region bounded by $\Gamma_{\beta}$, $\Gamma_{\gamma}$ and the circle $\{z:|z|=R\}$, and by $B:=B_{\gamma-\beta}$ the preimage of $A$ under $\psi$. Then we have the following formula for the (Euclidean) area of $A$ and $B$.

\begin{figure}[htbp] 
	\centering
	\includegraphics[width=13cm]{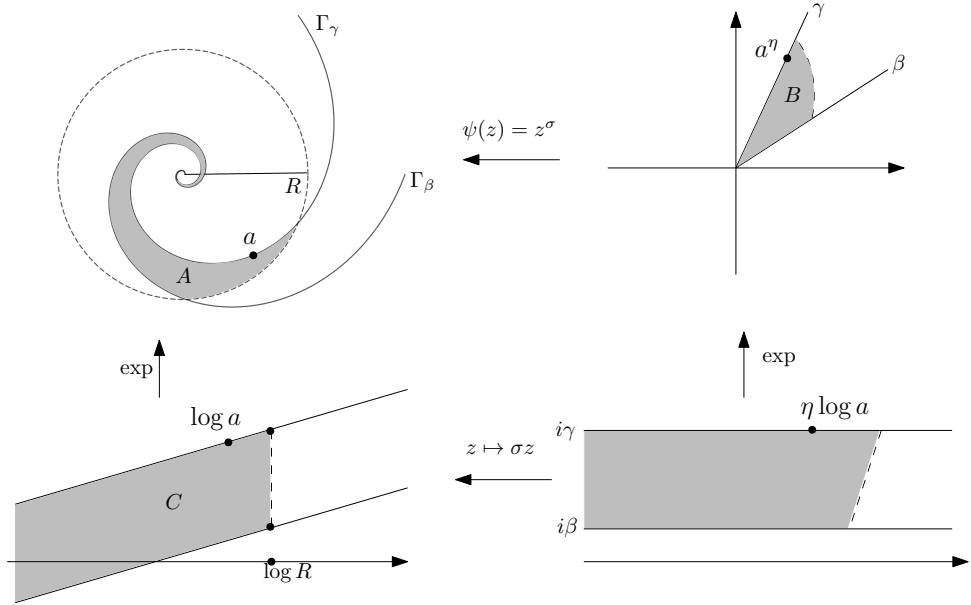}
	\caption{The area of the domain $A$ bounded by two logarithmic spirals can be computed by using a logarithmic change of variable.}
	\label{area}
\end{figure}
\begin{proposition}\label{aes}
\begin{align}
\area A&=\frac{(\gamma-\beta)|\sigma|^2}{2\re \sigma}R^2;\\
\area B&=\frac{\re \sigma}{4\Im\sigma}\left(e^{2\gamma\frac{\Im\sigma}{\re\sigma}} -e^{2\beta\frac{\Im\sigma}{\re\sigma}}\right)\cdot R^{\frac{2}{\re(\sigma)}}.
\end{align}
\end{proposition}

\begin{proof}
	Let $C$ be the shaded domain in Figure \ref{area}, which is mapped by the exponential map to $A$. Then
	$$\area A=\iint_{C}|(e^{z})'|^2|dz|^2=\int_{-\infty}^{\log R}e^{2x}\int_{\frac{\beta-x\Im(1/\sigma)}{\re(1/\sigma)}}^{\frac{\gamma-x\Im(1/\sigma)}{\re(1/\sigma)}}dydx=\frac{(\gamma-\beta)|\sigma|^2}{2\re \sigma}R^2.$$
	Similarly, by using the map $e^{z/\sigma}$ we can obtain the area formula for $B$. We omit details here.\qedhere
\end{proof}

We will also use in the sequel some standard notions and notations from Nevanlinna theory \cite{nevanlinna, goldbergmero, hayman1}. In particular, $n(r,f)$ denotes the number of poles of $f$ in the disk $\overline{D}(0,r)$, $m(r,f)$ the proximity function and $T(r,f)$ is the Nevanlinna characteristic of $f$. The order of a meromorphic function $f$ is defined by
$$\rho(f)=\limsup_{r\to\infty}\frac{\log T(r,f)}{\log r}.$$

By using Proposition \ref{aes}, we have the following estimate.

\begin{proposition}\label{order}
\begin{subequations}
  \begin{equation}\label{counting}
n(r,f)=\mathcal{O}\left(r^{\frac{\alpha^2+(\log|\tau|)^2}{\pi\alpha}}\right) ,
\end{equation}
  \begin{equation}\label{orderg}
\rho(f)=\frac{\alpha^2+(\log|\tau|)^2}{\pi\alpha}  \end{equation}
\end{subequations}
\end{proposition}

\begin{proof}
	To estimate the number of poles of $f$ in a certain disk, by \eqref{iden} it suffices to estimate the poles for the map $G$ for sufficiently large disk. This can be obtained by comparing the area of the disk and the area of a preimage of a parallelogram for Weierstra{\ss} elliptic function under the spiral function $h \circ q(z) = z^{\eta}$, where $\eta = \mu \alpha/2\pi$. Then it follows from Proposition \ref{aes} that
	$$n(r,G)\sim C\,r^{\frac{\alpha^2+(\log|\tau|)^2}{\pi\alpha}},$$
	for large $r$ and for some positive constant $C$. This gives \eqref{counting}.

To obtain \eqref{orderg}, we need a result of  Teichm\"uller \cite{teichmuller2} which states that for a meromorphic function $f\in\s$, if $\infty$ is not an asymptotic value and the multiplicities of poles are bounded, then $m(r,f)$ is bounded. This means that the order of $f$ can be estimated by using the integrated counting function $N(r,f)$ of $f$, where $N(r,f)=T(r,f)-m(r,f)$. Then classical results on the comparison on the growth scale of $N(r,f)$ and $n(r,f)$ shows that 
$$\rho(f)=\limsup_{r\to\infty}\frac{\log n(r,f)}{\log r}.$$
Now \eqref{orderg} follows from our estimate of $n(r,f)$.
\end{proof}

\medskip
In the above constructions, we have from the beginning fixed the parameters $\alpha$ and $\tau$. Now if we vary these parameters, we can achieve every finite order. To be more specific, we have the following.
\begin{proposition}\label{exf}
For any given $\rho\in(0,\infty)$, there exist $\alpha\in (0,2\pi]$ and $\tau$ satisfying \eqref{anre} such that
\begin{equation}\label{dfrho}
\rho=\frac{\alpha^2+(\log|\tau|)^2}{\pi\alpha}.
\end{equation}
Moreover, there is a meromorphic function $f\in\s_4$ such that $\rho(f)=\rho$.
\end{proposition}

Now we put
\begin{equation}\label{defeta}
\eta=\frac{\alpha\mu}{2\pi}=\frac{1}{2\pi}\left(\alpha-i\log|\tau|\right).
\end{equation}
Later on we will also need to estimate $|a^{\eta}|$ in terms of the modulus of $a$. Note that $\eta$ defined above is a complex number. We first prove the following result. For simplicity, we also use $\rho$ as given in \eqref{dfrho}.

\begin{proposition}\label{lengthestimate}
For $a\in\c$, we have
$$\left|a^{\eta} \right|=\mathcal{O}\left(|a|^{\rho/2}\right).$$
\end{proposition}

\begin{proof}
Let $S$ be the standard strip $S = \{ z= x + i y : 0 < y  < 2\pi \}$ and put $P = (1/\mu) S = \{ z :  \mu z \in S \}$. First, recall that $\exp$ maps the oblique strip $P$ onto $\c \setminus \Gamma$. It has a continuous extension to the boundary, by the definition of $\mu = 1-i \kappa$ (whose real part is equal to $1$). Here $\kappa=(\log|\tau|)/\alpha$. To make it injective, let us extend $\exp$ only to the lower boundary of the strip $P$, i.e. so that, for instance, the preimage of $z=1$ under exp has its preimage at the origin, i.e. the argument of this preimage is $0$. Then the argument of an arbitrary complex number $a \in \mathbb{C}$ is the imaginary part of the ``spiral branch'' $\log_P(a)$ of $a$; i.e. $\log_p$ is this particular inverse of $\exp$ mapping $\c \setminus \Gamma$ onto $P$.  

  By definition $a^{\mu} = e^{\mu \log a} = e^{(1-i\kappa) (\ln r + i \theta(r))}$, where $\theta=\theta(r)$ now depends on $r:=|a|$ in such a way that, if $\tilde{\theta} \in [0,2\pi)$ is the natural argument for $a$, then $\theta(r) = \tilde{\theta}  + 2\pi k $, where $k$ is the unique integer such that
  \[
\tilde{\theta} + 2\pi k \in [\kappa \log r, \kappa \log r + 2\pi). 
\]
Hence there is some $\delta \in [0,2\pi)$ such that
\[
k = \frac{\kappa \log r - \tilde{\theta} + \delta}{2\pi}.
\]
Replacing the expression for $k$ in $\theta(r)$, the computation goes, with $\kappa = (\log|\tau|)/\alpha$, 
\begin{equation}
 |a^{\eta}|= |a^{\mu \alpha/2\pi}| = r^{\alpha/2\pi} e^{\theta(r) \log|\tau|/2\pi} 
   = r^{\frac{\alpha^2 + (\log|\tau|)^2}{2\pi \alpha}} e^{\frac{\delta \log |\tau|}{2 \pi}}.\qedhere
\end{equation}
  \end{proof}

\bigskip
The strategy below follows in a similar way as we have done in \cite[Section 3]{ac1}. We first estimate asymptotic behaviours of the constructed function $f$ near its poles, which will be useful in estimating the Hausdorff dimension of escaping sets later on. Then we show that by varying parameters $\alpha$ and $\tau$ we obtain functions in $\s_4$ which are actually equivalent. This is crucial for completing the proof of Theorem \ref{incm}.

\medskip
\noindent{\emph{Local behaviours near poles}.} We first consider the local behaviours of the quasi-meromorphic map $G$ near its poles. In the following, $A\sim B$ means that $A$ and $B$ have comparable modulus. Let $z_0$ be a pole of $G$. By \eqref{defG}, $\Phi(h(q(z_0)))$ is a pole of the function $\wp$. Put $\zeta=\Phi(h(q(z)))$ and $\zeta_0=\Phi(h(q(z_0)))$. Since $\zeta_0$ is a pole of $\wp$, we see that there exists a constant $C$ such that
\begin{equation}\label{ppb}
\wp(\zeta)\sim\left(\frac{C}{\zeta-\zeta_0}\right)^2\,~\,~\text{as~}\,~\,\zeta\to\zeta_0.
\end{equation}
By the construction of $\Phi$, we have that
$$\zeta-\zeta_0\sim C'\left( h(q(z))-h(q(z_0))\right)=C'\left( z^{\alpha\mu/(2\pi)}-z_{0}^{\alpha\mu/(2\pi)} \right)\,~\,\text{as}\,~\,z\to z_0,$$
where $C'$ is some constant. This, together with \eqref{defG} and \eqref{ppb} shows that
\begin{equation}\label{Gpb}
G(z)\sim\left(\frac{C''}{z^{\alpha\mu/(2\pi)}-z_{0}^{\alpha\mu/(2\pi)}}\right)^2\,~\,~\text{as~}\,~\,z\to z_0,
\end{equation}
where $C''$ depends only on $C$ and $C'$. Note that $z_0$ is a double pole of $G$. We may thus assume that
\begin{equation}\label{rpb}
G(z)\sim\left(\frac{a(z)}{z-z_0}\right)^2\,~\,~\text{as~}\,~\,z\to z_0
\end{equation}
for some function $a$ which is holomorphic near in some neighbourhood of $z_0$ and moreover, $a(z_0)\neq 0$. By comparing \eqref{Gpb} and \eqref{rpb}, we see that
$$a(z)\sim C''\frac{z-z_0}{z^{\alpha\mu/(2\pi)}-z_{0}^{\alpha\mu/(2\pi)}}\,~\,~\text{as~}\,~\,z\to z_0.$$
Recall that $\eta$ is defined in \eqref{defeta}. Now by taking limit and using L'Hospital's rule we obtain
$$a(z_0)\sim C''' z_{0}^{1-\alpha\mu/(2\pi)}=C''' z_{0}^{1-\eta}.$$

Denoted by $w=\Phi(z)$ and $w_0=\Phi(z_0)$. Then $w_0$ is a double pole of $f$. Assume that
$$f(w)\sim\left( \frac{b(w)}{w-w_0} \right)^2\,~\,~\text{as~}\,~\,w\to w_0.$$
Here $b(w)$ is a function holomorphic in $w_0$ and $b(w_0)\neq 0$. So with \eqref{rpb} we see that
$$b(w)\sim a(z)\frac{w-w_0}{z-z_0}$$
as $z\to z_0$. Note that \eqref{iden} implies that $w\to w_0$ as $z\to z_0$. Again, \eqref{iden} will also imply that
$$b(w_0)\sim a(z_0)\lim_{z\to z_0}\frac{w-w_0}{z-z_0}=a(z_0).$$
Now it follows that near a pole $w_0$ of $f$, we have the following asymptotic relation:
\begin{equation}\label{localbehaviours}
f(w)=\left( \frac{C_1 w_{0}^{1-\eta}}{w-w_0} \right)^2~\,~\text{as~}\,~\,w\to w_0,
\end{equation}
where $C_1$ is some constant.


\medskip

\noindent{\emph{Equivalence.}} By choosing distinct parameters $(\alpha_i,\tau_i)$ we may reach that $\rho_i$ as defined in \eqref{orderg} are different for $i=1,2$. So by the construction we can have two functions $f_i\in\s_4$ whose orders are $\rho_i$. Following the idea of proof given in \cite[Theorem 3.2]{ac1}, we see that $f_1$ is quasiconformally equivalent to $f_2$. We omit this proof here but only state the result below. We leave the details for interested readers.

\begin{proposition}\label{eqdo}
$f_1$ is quasiconformally equivalent to $f_2$.
\end{proposition}

\bigskip
One may notice that the above constructions only give meromorphic functions in $\s_4$ of finite but non-zero order. To achieve zero or full Hausdorff dimension of escaping sets, we may need functions of zero or infinity order. We mention several examples in the following which suffices for our purposes.

\medskip
\noindent{\emph{Zero order.}} To have a zero order meromorphic function in $\s_4$, one can consider the one used in \cite[Section 4.1]{ac1}. Roughly speaking, the function $f$ is obtained by precomposing a suitably chosen Weierstra{\ss} elliptic function with an inverse branch of the function $\cosh$. The obtained function is meromorphic in certain slit plane. One can then use some basic properties of these two functions to show that it can be extended continuously across the slit and thus gives a meromorphic function in the plane. That this function belongs to $\s_4$ is clear since all singular values are just critical values of the Weierstra{\ss} elliptic function. It is also clear that $f$ has zero order by checking the counting function of poles, which can be computed explicitly. We omit details here.

\medskip
\noindent{\emph{Infinite order.}} As for infinite order functions in $\s_4$, consider the following function
$$f(z)=\wp(2\pi\cosh(z)),$$
where $\wp$ is a Weierstrass elliptic function with two periods $2\pi$ and $\tau$ such that $\tau$ is not a real multiple of $2\pi$. It follows from construction that $f$ has four critical values which are exactly the critical values of $\wp$ and has no asymptotic values. That $\rho(f)=\infty$ follows directly from \cite[Corollary 1.2]{edrei1}. This can also be obtained by checking the counting functions of poles.

\section{Estimate of the dimension}\label{section5}

We will need to estimate the Hausdorff dimension of the escaping sets for the Speiser functions in $\s_4$ constructed in the previous section, which then finishes the proof of Theorem \ref{main}. This will follow from the following estimate.

\begin{proposition}\label{mainpro}
Let $\rho\in [0,\infty)$. Then there exists a meromorphic function $f\in\s_4$ such that $\dim\I(f)=\frac{2\rho}{1+\rho}$. 
\end{proposition}

Before we prove this, we first state the following direct consequence.

\begin{corollary}
For any $d\in [0,2)$, there exists $f\in\s_4$ such that $\dim\I(f)=d$.
\end{corollary}

Therefore, Theorem \ref{main} is proved except for the case that $d=2$, which will be discussed in the last of this section. By combining this with Proposition \ref{eqdo}, Theorem \ref{incm} follows.

The rest of proof will be devoted to the proof of the above Proposition \ref{mainpro}. We only give a sketch of proof, as the estimate of the Hausdorff dimension for escaping sets for our constructed functions will use the same idea of proof in \cite{ac1}: the upper bound follows from a result of Bergweiler and Kotus \cite{bergweiler2} while the lower bound uses the spherical version of a well known result of McMullen \cite{mcmullen11}.

\begin{proof}[Proof of Proposition \ref{mainpro}]

For any given $\rho\in [0,\infty)$, there exist $\alpha$ and $\mu$ satisfying 
$$\rho=\frac{\alpha^2+(\log|\tau|)^2}{\pi\alpha}.$$
The constructions in Section $3$ then gives us a Speiser function $f$ in $\s_4$ whose order is $\rho$. This function $f$ has only critical values and no asymptotic values. Moreover, all poles have the multiplicity $2$. 

\smallskip
\noindent{\emph{Upper bound}}. It is clear that the above function $f$ satisfies the conditions of \cite[Theorem 1.1]{bergweiler2}. So we have
\begin{equation}\label{ubes}
	\dim\I(f)\leq \frac{2\rho}{1+\rho}.
\end{equation}

\noindent{\emph{Lower bound}}. Suppose that $a_j$ are the poles of $f$, where $\cdots\leq |a_j|\leq |a_{j+1}|\leq\cdots$. Then it follows from \eqref{localbehaviours} and Proposition \ref{lengthestimate} that
$$f(z)\sim \left(\frac{b_j}{z-a_j}\right)^2\,~\,\text{as}\,~\,z\to z_j,$$
where
$$|b_j|\sim \left|a_j^{1-\eta}\right|\sim |a_j|^{1-\rho/2}.$$

Since $f$ is Speiser, we may take a large $R_0>0$ such that $D(0,R_0)$ contains all singular values of $f$. Now with $B(R)=\hc\setminus \overline{D}(0,R)$ and $R>R_0$, each component of $f^{-1}(B(R))$ is bounded, simply connected and contains one pole of $f$. Let $U_j$ be the component containing the pole $a_j$. By using Koebe's distortion and one quarter theorem, we may obtain
\begin{equation}\label{bund}
D\left(a_j,\frac{|b_j|}{4\sqrt{R}} \right)\subset U_j\subset D\left(a_j,\frac{2|b_j|}{\sqrt{R}} \right).
\end{equation}
See \cite[Section 4]{ac1} for more details. Moreover, if $g_j$ is an inverse branch of $f$ from some domain $\Omega$ to $U_j$, where $\Omega\subset B(R)$, then
\begin{equation}\label{der}
|g'_{j}(z)|\leq B_1\,\frac{|b_j|}{|z|^{3/2}}\,~\,\text{for}~\,z\in \Omega.
\end{equation}
Here $B_1>0$ is some constant.

We denote by $\diam (E)$ the diameter of the set $E$ in the plane and $\diam_{\chi} (E)$ the spherical diameter of $E$. The above estimates \eqref{bund} and \eqref{der} will give us a good control over the sizes of the pullbacks of $U_k$ under $f$ for large $k$. More precisely, for sufficiently large $k$, we have
$$\diam g_j(U_k)\leq\sup_{z\in U_k}|g'_j(z)|\diam U_k \leq B_1\frac{|b_j|}{|a_k|^{3/2}}\frac{|b_k|}{\sqrt{R}},$$
and if the indices $j_1,\,\dots,\,j_{\ell}$ are chosen such that $U_{j_k}$ is contained in $B(R)$, where $k=1,\dots,\ell$, we obtain, in terms of spherical metric,
\begin{equation}\label{54}
\diam_{\chi}\left(g_{j_1}\circ g_{j_2}\circ\cdots\circ g_{j_{\ell-1}}\right)(U_{j_{\ell}})\leq B_{1}^{\ell-1}\,\frac{32}{\sqrt{R}}\,\prod_{k=1}^{\ell}\,\frac{|b_{j_k}|}{|a_{j_k}|^{3/2}}.
\end{equation}

Now we consider the set
$$\I_{R}(f)=\left\{z\in\I(f):\,f^{n}(z)\in B(R)~\,\text{for all}\,~n\in\mathbb{N}\,  \right\}.$$
In other words, we are considering those escaping points whose iterates always stay in $B(R)$. Apparently, this is a subset of $\I(f)$. Let $E_l$ be the collection of all components $V$ of $f^{-l}(B(R))$ for which $f^{k}(V)\subset B(R)$ holds for $0\leq k\leq l-1$. We are going to estimate the sizes of components of $E_l$ by using \eqref{54}. For such a component $V$, by definition there exists $j_1,\,\dots,\,j_{l-1}$ such that
$$f^{k}(V)\subset U_{j_{k+1}}\,~\,\text{for}\,~\,k=0,1,\dots,l-1.$$
So, using \eqref{54} one can have, for some constant $B_2$ and $B_3$,
\begin{equation}\label{55}
\diam_{\chi}\left(V\right)\leq B_{1}^{\ell-1}\,\frac{32}{\sqrt{R}}\,\prod_{k=1}^{\ell}\,\frac{|b_{j_k}|}{|a_{j_k}|^{3/2}}=B_{1}^{\ell-1}\,\frac{32}{\sqrt{R}}\,\prod_{k=1}^{\ell}\,\frac{|B_2|^{\ell}}{|a_{j_k}|^{1/2+\rho/2}}\leq \left(\frac{B_3}{|R|^{1/2+\rho/2}}\right)^{\ell}.
\end{equation}
Put
\begin{equation}\label{dl}
d_{\ell}=\left(\frac{|B_3|}{|R|^{1/2+\rho/2}}\right)^{\ell}.
\end{equation}
In addition to the above defined term $d_{\ell}$, McMullen's lower bound estimate of Hausdorff dimension also involves a lower bound estimate for the density of $\overline{E}_{l+1}$ in $V$. Here $\overline{E}_{\ell}$ represents the union of all elements of $E_{\ell}$. We will also define $E=\medcap_{\ell}\overline{E}_{\ell}$. For this purpose, we consider an annulus $A(s):=\{z: s<|z|<2s\}$ which is contained in $B(R)$, i.e., $s> R$. Then the number of $U_j$ contained in $A(s)$ is $B_4(n(2s,f)-n(s,f))=B_5 s^{\rho}$, where $B_4$ and $B_5$ are some positive numbers. So we have
$$\diam U_j\geq\frac{|b_j|}{2\sqrt{R}}=B_6\frac{|a_j|^{1-\rho/2}}{2\sqrt{R}}\geq \frac{B_7}{s^{\rho/2-1/2}}.$$
$B_6,\,B_7$ are constants. Therefore,
$$\dens\left(\overline{E}_1, A(s)\right)=\frac{\area(\overline{E}_1 \cap A(s))}{\area A(s)}\geq \frac{B_5 s^{\rho}\,\pi(B_7/s^{\rho/2-1/2})^2}{3\pi s^2}=B_8\, s,$$
where $B_8>0$ is some constant. By repeating the argument used in \cite[Section 4]{ac1}, which we do not repeat here, we have the following estimate, for some constant $B_9>0$,
\begin{equation}\label{del}
\dens_{\chi}\left(\overline{E}_{\ell+1}, V\right)\geq \frac{B_9}{R}=:\Delta_{\ell}.\end{equation}
Now we can apply McMullen's result by using \eqref{dl} and \eqref{del} to obtain
$$\dim E\geq 2-\limsup_{\ell\to\infty}\frac{\sum_{j=1}^{\ell+1}|\log \Delta_j|}{|\log d_{\ell}|}\geq 2-\frac{\log B_9 - \log R}{\log B_3 - (\frac{1}{2}+\frac{\rho}{2})\log R}.$$
With $R\to\infty$, we have $\dim E\geq \frac{2\rho}{1+\rho}$. The next step is to use this estimate  to  give the estimate for the Hausdorff dimension of the escaping set by taking a sequence $(R_k)$ which tends to infinity increasingly and consider those points whose $k$-iterate lies in in $B(R_k)$. This goes in the same way as in the aforementioned reference and so we omit details. We conclude directly that
$$\dim\I(f)\geq \frac{2\rho}{1+\rho}.$$
Combine with the upper bound discussed above, we have finished the proof.
\end{proof}

To complete the proof of Theorem \ref{main}, we still need to find a function in $\s_4$ with a full dimensional escaping set. For this purpose, we put
$$f_1(z)=\wp_1(2\pi\cosh(z))$$
This function belongs to $\s_4$, as mentioned before.

\begin{proposition}\label{fulldimension}
	$$\dim\I(f_1)=2.$$
\end{proposition}

The proof of this result is in the same manner as we did in \cite[Section 4.3]{ac1}. We leave details to interested readers.

\bigskip
\emph{Centre for Mathematical Sciences}

\emph{Lund University}

\emph{Box 118, 22 100 Lund, Sweden}
 
\medskip
\emph{magnus.aspenberg@math.lth.se}

\smallskip
\emph{weiwei.cui@math.lth.se}


\begin{thebibliography}{BFRG15}

\bibitem[AC21]{ac1}
M.~Aspenberg and W.~Cui.
\newblock Hausdorff dimension of escaping sets of meromorphic functions.
\newblock {\em Trans. Amer. Math. Soc.}, 374(9):6145--6178, 2021.

\bibitem[Ahl06]{ahlfors8}
L.~V. Ahlfors.
\newblock {\em Lectures on Quasiconformal Mappings}, volume~38 of {\em
  University Lecture Series}.
\newblock American Mathematical Society, Providence, RI, second edition, 2006.

\bibitem[Bar08]{baranski1}
K.~Bara\'nski.
\newblock Hausdorff dimension of hairs and ends for entire maps of finite
  order.
\newblock {\em Math. Proc. Cambridge Philo. Soc.}, 145(3):719--737, 2008.

\bibitem[BC21]{bergweiler-cui}
W.~Bergweiler and W.~Cui.
\newblock The Hausdorff dimension of Julia sets of meromorphic functions in the Speiser class.
\newblock Preprint arXiv: 2105.00938, 2021.

\bibitem[BE95]{bergweiler18}
W.~Bergweiler and A.~Eremenko.
\newblock On the singularities of the inverse to a meromorphic function of
  finite order.
\newblock {\em Rev. Mat. Iberoam}, 11(2):355--373, 1995.

\bibitem[BE17]{bergweiler23}
W.~Bergweiler and A.~Eremenko.
\newblock On the {B}ank-{L}aine conjecture.
\newblock {\em J. Eur. Math. Soc. (JEMS)}, 19(6):1899--1909, 2017.

\bibitem[Ber93]{bergweiler1}
W.~Bergweiler.
\newblock Iteration of meromorphic functions.
\newblock {\em Bull. Amer. Math. Soc.}, 29(2):151--188, 1993.

\bibitem[BF14]{branner3}
B.~Branner and N.~Fagella.
\newblock {\em Quasiconformal Surgery in Holomorphic Dynamics}.
\newblock Cambridge studies in advanced mathematics 141. Cambridge University
  Press, 2014.

\bibitem[BFRG15]{bergweiler6}
W.~Bergweiler, N.~Fagella, and L.~Rempe-Gillen.
\newblock Hyperbolic entire functions with bounded {Fatou} components.
\newblock {\em Comment. Math. Helv.}, 90(4):799--829, 2015.

\bibitem[Bis15]{bishop1}
C.~J. Bishop.
\newblock Constructing entire functions by quasiconformal folding.
\newblock {\em Acta Math.}, 214(1):1--60, 2015.

\bibitem[Bis17]{bishop5}
C.~J. Bishop.
\newblock Models for the {S}peiser class.
\newblock {\em Proc. Lond. Math. Soc. (3)}, 114(5):765--797, 2017.

\bibitem[BK12]{bergweiler2}
W.~Bergweiler and J.~Kotus.
\newblock On the {H}ausdorff dimension of the escaping set of certain
  meromorphic functions.
\newblock {\em Trans. Amer. Math. Soc.}, 364(10):5369--5394, 2012.

\bibitem[BKS09]{bergweiler10}
W.~Bergweiler, B.~Karpi\'nska, and G.~M. Stallard.
\newblock The growth rate of an entire function and the {H}ausdorff dimension
  of its {J}ulia set.
\newblock {\em J. London Math. Soc. (2)}, 80(3):680--698, 2009.

\bibitem[Cui21a]{cuiwei3}
W.~Cui.
\newblock Entire functions arising from trees.
\newblock {\em Sci. China Math.}, 64(10): 2231--2248, 2021.

\bibitem[Cui21b]{cuiwei2}
W.~Cui.
\newblock Hausdorff dimension of escaping sets of {N}evanlinna functions.
\newblock {\em Int. Math. Res. Not. IMRN}, 2021(15):11767--11781.

\bibitem[EF64]{edrei1}
A.~Edrei and W.~H.~J. Fuchs.
\newblock On the zeros of {$f(g(z))$} where {$f$} and {$g$} are entire
  functions.
\newblock {\em J. Analyse Math.}, 12:243--255, 1964.

\bibitem[EL92]{eremenko2}
A.~Eremenko and M.~Lyubich.
\newblock Dynamical properties of some classes of entire functions.
\newblock {\em Ann. Inst. Fourier (Grenoble)}, 42(4):989--1020, 1992.

\bibitem[Elf34]{elfving1}
G.~Elfving.
\newblock {\"Uber} eine {Klasse} von {Riemannschen Fl\"achen} und ihre
  {Uniformiserung}.
\newblock {\em Acta Soc. Sci. Fenn. (N. S.)}, 2(3):1--60, 1934.

\bibitem[GK86]{goldberg4}
L.~R. Goldberg and L.~Keen.
\newblock A finiteness theorem for a dynamical class of entire functions.
\newblock {\em Ergodic Theory Dynam. Systems}, 6(2):183--192, 1986.

\bibitem[GK18]{kotus2}
P.~Galazka and J.~Kotus.
\newblock Escaping points and escaping parameters for singly periodic
  meromorphic maps: {H}ausdorff dimensions outlook.
\newblock {\em Complex Var. Elliptic Equ.}, 63(4):547--568, 2018.

\bibitem[GO08]{goldbergmero}
A.~A. Goldberg and I.~V. Ostrovskii.
\newblock {\em Value Distribution of Meromorphic Functions}, volume 236 of {\em
  Translations of Mathematical Monographs}.
\newblock American Mathematical Society, Providence, RI, 2008.

\bibitem[Hay64]{hayman1}
W.~K. Hayman.
\newblock {\em Meromorphic Functions}.
\newblock Oxford Mathematical Monographs. Clarendon Press, Oxford, 1964.

\bibitem[LV73]{lehto1}
O.~Lehto and K.~I. Virtanen.
\newblock {\em Quasiconformal Mappings in the Plane}, volume 126.
\newblock Springer-Verlag, New York-Heidelberg, second edition, 1973.

\bibitem[McM87]{mcmullen11}
C.~T. McMullen.
\newblock Area and{ Hausdorff dimension of Julia} sets of entire functions.
\newblock {\em Trans. Amer. Math. Soc.}, 300(1):329--342, 1987.

\bibitem[Nev70]{nevanlinna}
R.~Nevanlinna.
\newblock {\em Analytic Functions}.
\newblock Translated from the second German edition by Phillip Emig. Die
  Grundlehren der mathematischen Wissenschaften, Band 162. Springer-Verlag, New
  York-Berlin, 1970.

\bibitem[RS10]{rempe11}
L.~Rempe and G.~M. Stallard.
\newblock Hausdorff dimensions of escaping sets of transcendental entire
  functions.
\newblock {\em Proc. Amer. Math. Soc.}, 138(5):1657--1665, 2010.

\bibitem[Sch07]{schubert}
H.~Schubert.
\newblock {\"Uber} die {Hausdorff-Dimension der Juliamenge von Funktionen
  endlicher Ordnung}.
\newblock {\em Dissertation, University of Kiel}, 2007.

\bibitem[Tei37]{teichmuller2}
O.~Teichm\"uller.
\newblock Eine {U}mkehrung des zweiten {H}auptsatzes der
  {W}ertverteilungslehre.
\newblock {\em Deutsche Math.}, 2:96--107, 1937.

\end{thebibliography}
\end{document}